\documentclass[11pt]{article}
\pdfoutput=1                            

\usepackage[utf8]{inputenc}				
\usepackage[english]{babel}				

\usepackage{fullpage}					
\usepackage[shortlabels]{enumitem}		

\usepackage{cite}						

\usepackage{calc}						

\usepackage{amsmath}					
\usepackage{colonequals}				

\usepackage[pagebackref=true,hidelinks]{hyperref}
\renewcommand*{\backrefalt}[4]{%
    \ifcase #1 \footnotesize{(Not cited.)}%
    \or        \footnotesize{(Cited on p.~#2)}%
    \else      \footnotesize{(Cited on pp.~#2)}%
    \fi}

\usepackage{booktabs}					
\usepackage{multirow}
\usepackage{siunitx}

\usepackage{tikz}  						
\graphicspath{{./graphics/}}

\usepackage[margin=10pt,font=small,labelfont=bf,labelsep=colon]{caption}
\usepackage[capitalise,nameinlink]{cleveref}

\usepackage{amsthm}
\usepackage{thmtools}

\def\qed{\rule[0pt]{5pt}{5pt}\par\medskip}
\renewcommand{\qedhere}{\hfill ~\qed}
\renewenvironment{proof}{{\noindent\bf Proof.}}{\qedhere}

\usepackage{comment}
\usepackage[math]{blindtext}

\usepackage{parskip}

\usepackage{cite}                       
\usepackage{amsmath,amssymb,amsfonts}   
\usepackage{graphicx}                   
\usepackage{colonequals}                
\usepackage{arydshln}                   

\newtheorem{theorem}{Theorem}
\newtheorem{proposition}{Proposition}
\newtheorem{definition}{Definition}
\newtheorem{lemma}{Lemma}
\newtheorem{remark}{Remark}

\renewcommand{\epsilon}{\varepsilon}
\newcommand{\real}{\mathbb{R}}
\newcommand{\complex}{\mathbb{C}}
\newcommand{\1}{\mathbf{1}}
\newcommand{\0}{\mathbf{0}}
\newcommand{\G}{\mathbf{G}}
\renewcommand{\H}{\mathbf{H}}
\newcommand{\defeq}{\colonequals}
\newcommand{\diag}{\text{diag}}
\newcommand{\tp}{\mathsf{T}}
\newcommand{\bmat}[1]{\begin{bmatrix}#1\end{bmatrix}}
\newcommand{\grad}{\nabla\! }
\newcommand{\norm}[1]{\|#1\|}

\newcommand{\Gopt}{\ensuremath{G_\mathrm{opt}}}
\newcommand{\Gcon}{\ensuremath{G_\mathrm{con}}}
\newcommand{\Gconleft}{\ensuremath{G_\mathrm{con1}}}
\newcommand{\Gconright}{\ensuremath{G_\mathrm{con2}}}

\newcommand{\Goptz}{\ensuremath{\widehat{G}_\mathrm{opt}}}
\newcommand{\Gconz}{\ensuremath{\widehat{G}_\mathrm{con}}}
\newcommand{\Gconleftz}{\ensuremath{\widehat{G}_\mathrm{con1}}}
\newcommand{\Gconrightz}{\ensuremath{\widehat{G}_\mathrm{con2}}}

\newcommand{\GOPT}{\ensuremath{\mathbf{G}_\mathbf{opt}}}
\newcommand{\GCON}{\ensuremath{\mathbf{G}_\mathbf{con}}}

\newcommand{\df}{\ensuremath{\mathbf{\nabla f}}}
\renewcommand{\L}{\ensuremath{\mathbf{L}}}

\renewcommand{\u}{\mathbf{u}}
\renewcommand{\v}{\mathbf{v}}
\newcommand{\e}{\mathbf{e}}
\newcommand{\w}{\mathbf{w}}
\newcommand{\x}{\mathbf{x}}
\newcommand{\y}{\mathbf{y}}
\newcommand{\z}{\mathbf{z}}

\def\algorithmverticalspace{2mm}

\newcommand\Tstrut{\rule{0pt}{2.6ex}}         
\newcommand\Bstrut{\rule[-1.2ex]{0pt}{0pt}}   

\usepackage[bottom]{footmisc}   

\begin{document}

\title{A Universal Decomposition for\\Distributed Optimization Algorithms}

\author{
  Bryan Van Scoy\thanks{B.~Van~Scoy is with the Department of Electrical and Computer Engineering, Miami University, OH~45056, USA. Email \texttt{bvanscoy@miamioh.edu}}
\and
  Laurent Lessard\thanks{L.~Lessard is with the Department of Mechanical and Industrial Engineering, Northeastern University, MA~02115, USA. Email \texttt{l.lessard@northeastern.edu}}
}
\date{}
\maketitle


\begin{abstract}
  In the distributed optimization problem for a multi-agent system, each agent knows a local function and must find a minimizer of the sum of all agents' local functions by performing a combination of local gradient evaluations and communicating information with neighboring agents. We prove that every distributed optimization algorithm can be factored into a centralized optimization method and a second-order consensus estimator, effectively separating the ``optimization'' and ``consensus'' tasks. We illustrate this fact by providing the decomposition for many recently proposed distributed optimization algorithms. Conversely, we prove that any optimization method that converges in the centralized setting can be combined with any second-order consensus estimator to form a distributed optimization algorithm that converges in the multi-agent setting. Finally, we describe how our decomposition may lead to a more systematic algorithm design methodology.
\end{abstract}



\section{Introduction}

We consider the distributed optimization problem
\begin{alignat*}{2}
  &\text{minimize}   \quad && \sum_{i=1}^n f_i(x_i) \\
  &\text{subject to} \quad && x_1 = x_2 = \ldots = x_n,
\end{alignat*}
where $f_i$ is the local objective function and $x_i$ the local decision variable on agent $i\in\{1,\ldots,n\}$. The problem is to minimize the sum of the local objective functions subject to agreement among the agents on the solution, where each agent $i\in\{1,\ldots,n\}$ can evaluate its local gradient $\grad  f_i$ and can communicate with (and only with) neighboring agents\footnote{Some authors refer to this as the \textit{consensus optimization problem} and to such algorithms as \textit{decentralized}.}.

Many distributed algorithms have been proposed in the literature, and several recent works have attempted to uncover an underlying algorithmic structure. The work \cite{Unification} developed a framework that unified the EXTRA \cite{EXTRA} and DIGing \cite{DIGing} algorithms. The work \cite{canform} found a canonical form for distributed algorithms that encompasses cases where each agent has two state variables. This canonical form, however, was limited to non-accelerated algorithms. To handle acceleration, Han \cite{shuo} showed that an optimization method could be combined with two first-order consensus estimators to form a valid distributed optimization algorithm. This structure captures many algorithms, but not all algorithms decompose into this form. This is particularly important when the structure is used for design, as a limited structure may lead to suboptimal performance. For example, the structure in \cite{shuo} is unable to represent the SVL algorithm \cite{SVL}, which is an optimized implementation of inexact ADMM (see Section \ref{sec:non-accelerated}).

Our main result overcomes the aforementioned limitations and shows that a broad class of distributed optimization algorithm can be decomposed into a centralized optimization method and a second-order consensus estimator as shown in Figure \ref{fig:decomp}. Specifically, our decomposition applies to algorithms that are linear time-invariant (LTI) systems in feedback with the gradient of the objective function and the graph Laplacian. Conversely, we show that any centralized optimization method can be combined with any second-order consensus estimator to form a distributed optimization algorithm (under mild technical conditions).

\begin{figure}[h]
  \centering\includegraphics{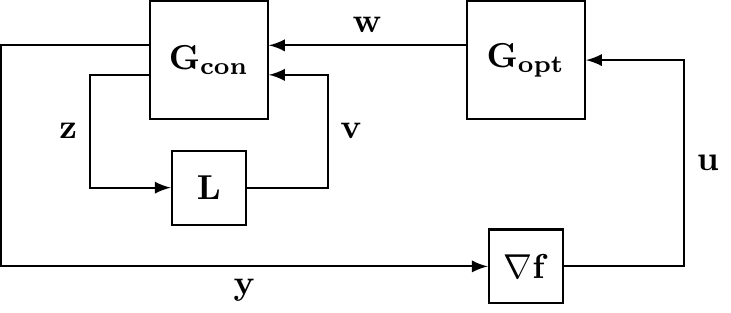}
  \caption{Universal decomposition of a distributed optimization algorithm into an optimization method $\GOPT$ and second-order consensus estimator $\GCON$, where $\df$ is the gradient of the local objective functions and $\L$ is the graph Laplacian.}
  \label{fig:decomp}
\end{figure}

Our decomposition has several benefits. First, it provides a non-conservative parameterization of distributed optimization algorithms in terms of their components, which can then be systematically analyzed using tools from robust control; see \cite{lessard,sundararajan_allerton,SVL} for the details of such analyses. Our decomposition also assists algorithm designers by simplifying the taxonomy of distributed optimization algorithms. Simply put, one need not look any further than the already vast literature on gradient-based optimization methods \cite{tmm,nesterov-book,polyak1964} and consensus estimators \cite{consensus,dac,fast-consensus}.

\paragraph{Notation.} Subscripts denote an agent's index, and bold symbols to refer to quantities aggregated over all agents, such as $\x = (x_1,\ldots,x_n)$. Superscripts denote time indices, as $\{y^0,y^1,\dots\}$. Symbols $\1$ and $\0$ denote the $n$-dimensional vector of all ones and all zeros, respectively. The symbol $\otimes$ denotes the Kronecker product. For an LTI operator $G$, the corresponding transfer function is  $\widehat G(z)$. A transfer function is \emph{stable} if all of its poles are in the open unit disk.

\paragraph{Assumption.} To simplify notation, we assume the local objective functions are one-dimensional, ${f_i : \real\to\real}$. Our results generalize to the multidimensional case ${f_i:\real^d\to\real}$ under appropriate restrictions on the algorithm form.

\section{Preliminaries}

Before describing algorithms for distributed optimization, we first describe optimization and consensus separately. We make extensive use of the Final Value Theorem (FVT), which we state here for completeness (e.g., \cite[pp.~2-12, 2-15]{levine2018control}).

\begin{proposition}[Final Value Theorem]
  Suppose $y^t$ has the unilateral $z$-transform $\hat y(z)$. The following are equivalent.
  \begin{itemize}
    \item The limit of $y^t$ as $t\to\infty$ exists and is finite.
    \item $(z-1)\,\hat y(z)$ is stable.
  \end{itemize}
  If the above hold, then $\lim_{t\to\infty} y^t = \lim_{z\to 1} (z-1)\,\hat y(z)$.
\end{proposition}

\subsection{Optimization}\label{sec:opt}

A gradient-based optimization method is an iterative procedure used to find an extremum of some function $f$ by sequentially querying $\grad f$. We can view such a method as a discrete-time dynamical system $\Gopt$ in feedback with $\grad f$ \cite{lessard}.

\noindent\begin{minipage}{0.6\linewidth}
  \vspace{\algorithmverticalspace}
  \centering\includegraphics{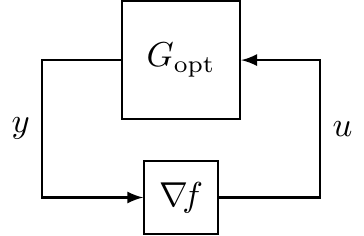}
  \vspace{\algorithmverticalspace}
\end{minipage}%
\begin{minipage}{0.4\linewidth}
  \begin{align*}
    y &= \Gopt\,u \\[5pt]
    u &= \grad  f(y)
  \end{align*}
\end{minipage}

For example, standard gradient descent uses the update $x^{t+1} = x^t - \alpha\,\grad  f(x^t)$, for which $\Gopt$ can be represented using the discrete-time transfer function $\Goptz(z) = \frac{-\alpha}{z-1}$. Methods such as gradient descent are \emph{strictly causal} and have strictly proper transfer functions.

If a method is causal but not strictly causal, then the feedback loop has a circular dependency. An example of such an algorithm is the \emph{proximal point method}, which uses the update $x^{t+1} \in \arg\min_{x} \left( f(x) + \frac{1}{2\alpha}\norm{x-x_t}^2 \right)$. The circular dependency is apparent when we write the associated first-order optimality condition: $x^{t+1} = x^t - \alpha\,\grad f(x^{t+1})$. This method has a proper transfer function: $\Goptz(z) = \frac{-\alpha z}{z-1}$.

In this letter, we define an \emph{optimization method} as a system $\Gopt$ that has the correct fixed point when placed in feedback with $\grad f$ and also exhibits convergent behavior when using a baseline set of \emph{easy} test functions $f$.

\begin{definition}\label{def:opt}
  Consider the feedback interconnection of a system $\Gopt$ with the gradient $\grad f$, where ${f(y) \defeq \frac{\epsilon}{2} \norm{y-y^\star}^2}$. The system $\Gopt$ is an \emph{optimization method} if for all $\epsilon > 0$ sufficiently small and for all $y^\star$, we have $y^t\to y^\star$ as $t\to\infty$, and $y^t$ converges to a constant when $\epsilon = 0$.
\end{definition}

If $\Gopt$ is causal, SISO, and LTI as with gradient descent and many other methods, we can characterize optimization methods via properties of the transfer function $\Goptz(z)$.

\begin{lemma}\label{lem:opt}
  A causal SISO LTI system $\Gopt$ is an optimization method if and only if the following hold:
  \begin{enumerate}
  \item[(i)] The zeros of $1-\epsilon\,\Goptz(z)$ are inside the unit circle for all $\epsilon>0$ sufficiently small.
  \item[(ii)] $\Goptz(z)$ has a pole at $z=1$ and $(z-1)\,\Goptz(z)$ is~stable.
  \item[(iii)] $\Goptz(z)$ is proper.
  \end{enumerate}
\end{lemma}

\begin{proof}
  Substituting the given $f$ and eliminating $u$, the closed-loop dynamics are $y = \frac{-\epsilon \Gopt}{1-\epsilon \Gopt} y^\star$. Applying the FVT, $y^t$ converging is equivalent to stability of the map and the zeros of $1-\epsilon\,\Goptz(z)$ being inside the unit circle. The limit $y^t\to y^\star$ is equivalent to $\Goptz(z)$ having a pole at $z=1$. For $\epsilon=0$, convergence to a constant is equivalent to stability of $(z-1)\,\Goptz(z)$. Causality of the system $\Gopt$ is equivalent to properness of the transfer function $\Goptz(z)$.
\end{proof}

\subsection{Consensus}\label{sec:con}

Consider a network of $n$ agents. Agent $i\in\{1,\ldots,n\}$ observes a time-varying signal $w_i$. A consensus estimator \cite{consensus,plp} is an iterative procedure where each agent communicates with its neighbors in order to form an estimate $y_i$ of the average $\frac{1}{n}\sum_{i=1}^n w_i$. Such estimators take the following form.

\noindent\begin{minipage}{0.55\linewidth}
  \vspace{\algorithmverticalspace}
  \centering\includegraphics{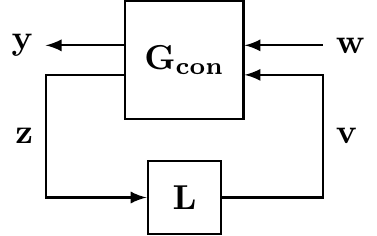}
  \vspace{\algorithmverticalspace}
\end{minipage}%
\begin{minipage}{0.45\linewidth}
  \begin{align*}
    \bmat{y_i \\ z_i} &= \Gcon \bmat{w_i \\ v_i} \\[5pt]
    v_i &= \sum_{j=1}^n a_{ij}\,(z_i-z_j)
  \end{align*}
\end{minipage}

The $n\times n$ matrix $A \defeq [a_{ij}]$ is the \textit{adjacency matrix} that describes the interaction among agents. The scalar $a_{ij}$ is the weight that agent $i$ places on information from agent $j$, with a weight of zero if no information flows from agent $j$ to~$i$. Agent $j$ is a \textit{neighbor} of agent $i$ if the weight $a_{ij}$ is nonzero, and computing $v_i$ requires agent $i$ to receive the local variables $z_j$ from each of its neighbors~$j$. The \textit{Laplacian} is the matrix $L \defeq \text{diag}(A\1)-A$. This matrix always satisfies $L\1 = \0$, so it has an eigenvalue of zero with corresponding eigenvector $\1$. When the communication network is \textit{connected}, meaning that there is a path between any two agents, there is exactly one zero eigenvalue  \cite{consensus}. When the weights are constructed such that $L^\tp\1 = \0$, the Laplacian is \textit{balanced} \cite{plp}.
The block diagram illustrates the global behavior of the system aggregated over all agents, where the aggregated system and Laplacian are
\[
  \GCON \defeq \bmat{I_n\otimes \Gcon^{11} & I_n\otimes \Gcon^{12} \\
    I_n\otimes \Gcon^{21} & I_n\otimes \Gcon^{22} \Tstrut}
    \quad\text{and}\quad
    \L \defeq L\otimes I_m,
\]
where $m$ is the dimension of the local vectors $v_i$ and $z_i$. For example, one particular (first-order) consensus estimator is given by the iterations $\x^{t+1} = \x^t + L\,(\w^t - \x^t)$ and $\y^t = \w^t - \x^t$, for which $\Gcon$ can be represented using the discrete-time transfer function
\begin{equation}\label{eq:estimator}
  \Gconz(z) = \bmat{1 & \frac{-1}{z-1} \\ 1 & \frac{-1}{z-1} \Tstrut}.
\end{equation}

We are interested in tracking signals with constant mean but potentially \emph{higher-order} deviations from the mean. We define a first-order estimator to have zero steady-state error for constant deviations from the mean, a second-order estimator for ramp deviations, and so on. Similar to the motivation for our definition of optimization methods, we define a \emph{consensus estimator} as a system $\Gcon$ that can successfully track the average when using a baseline set of \emph{easy} Laplacians~$L$.

\begin{definition}\label{def:consensus}
  A system $\Gcon$ is a \emph{consensus estimator of order $\ell$} if, for any connected communication network and associated balanced and diagonalizable Laplacian $L$ with spectral radius sufficiently small and for any signals $w_i^t$ that are polynomials in $t$ of degree $\ell-1$ with constant mean $w^\star \defeq \frac{1}{n}\sum_{j=1}^n w_j^t$, the estimate $y_i^t$ on each agent $i$ converges to the average $w^\star$ as $t\to\infty$.
\end{definition}

If $\Gcon$ is causal and LTI, we can characterize consensus estimators via properties of the transfer function $\Gconz(z)$. It is also typical to assume $\Gcon^{22}$ is strictly causal to avoid circular dependencies in the network transmissions.

\begin{lemma}\label{lem:consensus}
  Suppose $\Gcon$ is a causal LTI system and $\Gcon^{22}$ is strictly causal. For all complex $\lambda\in\complex$, define the map $G_\lambda \defeq \Gcon^{11} + \lambda\,\Gcon^{12} \bigl(I-\lambda\,\Gcon^{22}\bigr)^{-1} \Gcon^{21}$. Then, $\Gcon$ is a consensus estimator of order $\ell$ if and only if the following~hold:
  \begin{enumerate}
    \item[(i)] $\widehat G_\lambda(z)$ is stable for all $\lambda \in \complex$ satisfying $|\lambda| < \delta$ for some $\delta > 0$ sufficiently small.
    \item[(ii)] $\widehat G_0(1)=1$.
    \item[(iii)] $\widehat G_\lambda(z)$ has $\ell$ zeros at $z=1$ for all $\lambda\neq 0$.
    \item[(iv)] $\Gconz(z)$ is proper and $\Gconz^{22}(z)$ is strictly proper.
  \end{enumerate}
\end{lemma}

\begin{proof}
  We can write the error $e_i^t \defeq y_i^t - w^\star$ succinctly as $\e = (\G_\L-\frac{1}{n}\1\1^\tp)\,\w$, where $\G_\L$ is the closed-loop map from $\w$ to $\y$. Let $(\lambda,v^\tp)$ be a left eigen-pair of $L$. Using that $L\1=\0$, we have that $0 = v^\tp L \1 = \lambda\,(v^\tp \1)$. Thus, $\1^\tp v=0$ for all $v$ corresponding to nonzero $\lambda$. Furthermore, $L^\tp\1=\0$ since the Laplacian is balanced (by assumption). The inner product of an eigenvector with the error is then
  \[
    v^\tp \e = \begin{cases}
    (G_0 - 1) (\1^\tp \w) & \text{if $v=\1$ and $\lambda=0$} \\
    G_\lambda (v^\tp \w) & \text{otherwise.}
    \end{cases}
  \]
  Since the Laplacian is diagonalizable (by assumption), convergence of the error $\e$ is equivalent to convergence of $v^\tp \e$ for each eigenvector $v$. Applying the FVT in the case where $w_i^t$ are polynomials in $t$ of degree $\ell-1$ with constant average $w^\star$, the limit $e_i^t\to 0$ is equivalent to stability of $\widehat G_\lambda$ for all eigenvalues $\lambda$ of the Laplacian $L$ and
  \[
    \lim_{z\to 1} \ \widehat G_0(z)-1= 0
    \quad\text{and}\quad
    \lim_{z\to 1} \ \frac{\widehat G_\lambda(z)}{(z-1)^{\ell-1}} = 0,
  \]
  which correspond to the first three conditions. Causality of $\Gcon$ and strict causality of $\Gcon^{22}$ are equivalent to properness and strict properness of $\Gconz(z)$ and $\Gconz^{22}(z)$, respectively.
\end{proof}

\begin{remark}\label{rem:Ftransform}
  The transfer function of a consensus estimator is not unique. Let $\widehat F$ be any $m\times m$ transfer matrix with full normal rank, where $m$ is the dimension of $v_i$ and $z_i$. Then the closed-loop map $\G_\L:\w\mapsto \y$ is invariant under
  \[
    \Gconz \mapsto \bmat{1 & 0 \\ 0 & \widehat F} \Gconz \bmat{1 & 0 \\ 0 & \widehat F^{-1}},
  \]
  although not all choices of $\widehat F$ preserve causality of $\Gcon$.
\end{remark}

\subsection{Distributed optimization}\label{sec:optcon}

The distributed optimization setting is conceptually a combination of the optimization and consensus settings. There are $n$ agents that can communicate over a network, agent $i$ has access to the gradient of its local function $\grad f_i$, and the goal is for all agents to achieve consensus on an extremum of the sum of all functions $f_1+\dots+f_n$. Distributed optimization algorithms take the following general form \cite{sundararajan_allerton,SVL}; see Section~\ref{sec:examples} for specific examples from the literature.

\noindent\begin{minipage}{0.55\linewidth}
  \vspace{\algorithmverticalspace}
  \centering\includegraphics{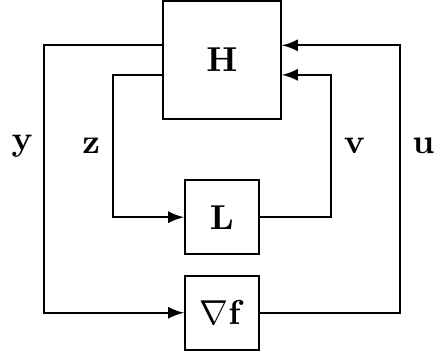}
  \vspace{\algorithmverticalspace}
\end{minipage}\hfill%
\begin{minipage}{0.45\linewidth}
  \begin{align*}
    \bmat{y_i \\ z_i} &= H \bmat{u_i \\ v_i}, \\[5pt]
    u_i &= \grad  f_i(y_i), \\[5pt]
    v_i &= \sum_{j=1}^n a_{ij}\,(z_i-z_j)
  \end{align*}
\end{minipage}

\noindent Similar to the consensus and optimization settings, we have 
\[
  \H \defeq
    \addtolength{\arraycolsep}{-1mm}
    \bmat{I_n\otimes H^{11} & I_n\otimes H^{12} \\
    I_n\otimes H^{21} & I_n\otimes H^{22}}\!,\quad
    \begin{aligned}
    \L &\defeq L\otimes I_m,\\
    \df &\defeq\diag(\grad f_1,\dots,\!\grad f_n).
    \end{aligned}
\]
We define a distributed optimization algorithm as follows.

\begin{definition}\label{def:distrop}
  A system $H$ is a \emph{distributed optimization algorithm} if for any connected communication network and associated balanced Laplacian $L$ with spectral radius sufficiently small, and for all $\epsilon>0$ sufficiently small and for all $y_{i}^\star$, the feedback interconnection of $\H$ with $\L$ and $\df$ satisfies
  $y_i^t \to y^{\star} \defeq \frac{1}{n}\sum_{j=1}^n y_j^\star$ as $t\to\infty$ for all $i$,  where $f_i(y) \defeq \frac{\epsilon}{2} \norm{y-y_i^\star}^2$. We also require that all $y_i^t$ converge to a common constant limit when $f_i\equiv 0$ for all $i$.
\end{definition}

If $H$ is causal and LTI, we can characterize consensus estimators via properties of the transfer function $\widehat H(z)$. We will also assume causality of certain maps to ensure that the algorithm is implementable. In particular, $H$ should be causal, and there should be no circular dependencies in the network transmissions or gradient evaluations. This means that the partial closed-loop map $H^{22} + \epsilon H^{21}(I - \epsilon H^{11})^{-1} H^{12}$ should be strictly causal, which is equivalent to both $H^{22}$ and $H^{21} H^{12}$ being strictly causal.

\begin{lemma}\label{lem:distrop}
  Suppose $H$ is a causal LTI system, and $H^{22}$ and $H^{21}H^{12}$ are strictly causal. For all $\lambda\in\complex$, define the map
  \[
    H_\lambda \defeq H^{11} + \lambda H^{12} \bigl(I-\lambda H^{22}\bigr)^{-1}H^{21}.
  \]
  The system $H$ is a distributed optimization algorithm if and only if the following hold:
  \begin{enumerate}
  \item[(i)] The zeros of $1-\epsilon \widehat H_\lambda(z)$ are inside the unit circle for all $\epsilon > 0$ sufficiently small and for all $\lambda\in\complex$ satisfying $|\lambda| < \delta$ for some $\delta > 0$ sufficiently small.
  \item[(ii)] $\widehat H_0(z)$ has a pole at $z=1$ and $(z-1)\,\widehat H_0(z)$ is stable.
  \item[(iii)] $\widehat H_\lambda(z)$ is stable and has a zero at $z=1$ for all $\lambda\neq 0$.
  \item[(iv)] $\widehat H(z)$ is proper and both $\widehat H^{22}(z)$ and $\widehat H^{21}(z)\,\widehat H^{12}(z)$ are strictly proper.
  \end{enumerate}
\end{lemma}

\begin{proof}
  Let $\H_\L$ be the partial closed-loop map from $\u$ to $\y$ after we eliminate $\z$ and $\v$. Substituting the given $f_i$ and eliminating $\u$, we obtain the closed-loop dynamics $\y = -\epsilon \H_\L(I - \epsilon \H_\L)^{-1} \y^\star$. The condition $y_i^t \to \frac{1}{n}\sum_{j=1}^n y_j^\star$ can be written succinctly as $\y^t \to \frac{1}{n}\1\1^\tp \y^\star$. Diagonalizing the closed-loop dynamics as in the proof of Lemma~\ref{lem:consensus} and applying the FVT, we find that $y_i^t$ converging is equivalent to the map $-\epsilon \widehat H_\lambda(z)(I-\epsilon\widehat H_\lambda(z))^{-1}$ being stable, which is equivalent to~(i). Again from the FVT, $\1^\tp \y^t \to \1^\tp \y^\star$ means $H_0(z)$ has a pole at $z=1$, and convergence to a constant in the case $f_i\equiv 0$ means $(z-1)\,\widehat H_0(z)$ is stable, so we have (ii). As in the proof of Lemma~\ref{lem:consensus}, we have $v^\tp \1 = 0$ and $v^\tp \y^t \to 0$ for all $v$ corresponding to $\lambda\neq 0$, so $\widehat H_\lambda(z)$ has a zero at $z=1$, and for the case $f_i\equiv 0$, we have that $\widehat H_\lambda(z)$ is stable, which is equivalent to (iii). Item (iv) is equivalent to the causality assumptions.
\end{proof}

\section{Universal decomposition}

We now state our main result, which states that every distributed optimization algorithm can be decomposed into consensus and optimization components as in Figure~\ref{fig:decomp}.

\begin{theorem}\label{thm:1}
  Let $H$ be a distributed optimization algorithm satisfying the conditions of Lemma~\ref{lem:distrop}. There exists an optimization method $\Gopt$ and a second-order consensus estimator $\Gcon$ such that
  \begin{equation}\label{Htrans}
  H = \Gcon \bmat{\Gopt & 0 \\ 0 & I_m}.
  \end{equation}
  If $H^{11}$ is strictly causal, then $\Gopt$ can be chosen to be strictly causal as well.
\end{theorem}

\begin{proof}
  From conditions (i), (ii), and (iv) of Lemma~\ref{lem:distrop}, $\widehat H_0(z)$ has a pole at $z=1$ and is proper, $(z-1)\,\widehat H_0(z)$ is stable, and the zeros of $1-\epsilon\,\widehat H_0(z)$ are inside the unit circle for all $\epsilon > 0$ sufficiently small. Then from Lemma~\ref{lem:opt}, $H_0 = H^{11}$ is an optimization method. If $\widehat H^{11}(z)$ is non-minimum phase (has zeros on or outside the unit circle), then $\Goptz(z) \defeq z^p \prod\left( \frac{1-\bar z_0 z}{z-z_0}\right) \widehat H^{11}(z)$, where the product is over all such zeros $z_0$, will also satisfy the conditions of Lemma~\ref{lem:opt}, provided $p$ is at most the relative degree of $\widehat H^{11}(z)$. This follows because $\Goptz(z)$ is still proper, still has a pole at $z=1$, and because each factor multiplying $\widehat H^{11}(z)$ is an all-pass filter with nonnegative phase (phase lead), which therefore can only increase stability margins and preserves the stability requirement.

  Set $\Goptz(z) = z^p\,\widehat\Phi(z)\,\widehat H^{11}(z)$, where $\widehat\Phi(z) \defeq \prod\frac{1-\bar z_0 z}{z-z_0}$ is the product of all-pass factors that cancel the non-minimum phase zeros of $\widehat H^{11}(z)$. Then, invert the transformation \eqref{Htrans} and apply Remark~\ref{rem:Ftransform} using $\widehat F(z) = z^{-q} I$ to obtain
  \[
    \Gconz(z) = \bmat{ z^{-p}\,\widehat\Phi(z)^{-1} & z^{q}\,\widehat H^{12}(z) \\ z^{-p-q}\,\widehat H^{21}(z)\,\widehat H^{11}(z)^{-1}\,\widehat\Phi(z)^{-1} & \widehat H^{22}(z)}.
  \]
  Since $\widehat H^{11}(z)$ is proper and $\widehat H^{21}(z)\,\widehat H^{12}(z)$ is strictly proper, we can always ensure $\Gconz(z)$ will be proper by letting $p$ and $q$ be the relative degrees of $\widehat H^{11}(z)$ and $\widehat H^{12}(z)$, respectively. This choice leads to a $\Goptz(z)$ that has relative degree zero. However, when $\widehat H^{11}(z)$ is strictly proper, we can reduce $p$ by~$1$, which ensures that $\Goptz(z)$ is strictly proper as well.

  To verify that $\Gcon$ is a consensus estimator of order two, we can compute $G_\lambda$ as defined in Lemma~\ref{lem:consensus} and see that $\widehat G_\lambda(z) = \bigl( z^p \widehat H^{11}(z)\,\widehat\Phi(z) \bigr)^{-1} \widehat H_\lambda(z)$, where $H_\lambda$ is defined in Lemma~\ref{lem:distrop}. We can now verify the properties in Lemma~\ref{lem:consensus}. When $\lambda=0$, the transfer function is $\widehat G_0(z) = z^{-p}\,\widehat\Phi(z)^{-1}$, which is stable and satisfies $\widehat G_0(1)=1$ since $\widehat\Phi$ is all-pass. When $\lambda \neq 0$, the fact that $\widehat H_\lambda(z)$ has a zero at $z=1$ and $\widehat H_0(z) = \widehat H^{11}(z)$ has a pole at $z=1$ implies that $\widehat G_\lambda(z)$ has two zeros at $z=1$. To verify stability when $\lambda \neq 0$, stability of $\widehat G_\lambda(z)$ follows from stability of $\widehat H_\lambda(z)$ and $\widehat\Phi(z)^{-1}$.
\end{proof}

We can also prove a partial converse; under certain mild technical conditions, combining consensus and optimization components as in Figure~\ref{fig:decomp} yields a distributed optimization algorithm.

\begin{theorem}\label{thm:2}
  Suppose $\Gopt$ is a causal SISO LTI optimization method, $\Gcon$ is a causal LTI second-order consensus estimator, and further assume
  \begin{itemize}
    \item $\Goptz(z)$ and $\widehat G_\lambda(z)$ are minimum-phase, meaning all zeros are strictly inside the unit circle, and
    \item either $\Goptz(z)$ or $\Gconz^{21}(z)\,\Gconz^{12}(z)$ is strictly proper.
  \end{itemize}
  Then, the combined system $H$ given in \eqref{Htrans} is a distributed optimization algorithm.
\end{theorem}

\begin{proof}
  We will verify the properties of Lemma~\ref{lem:distrop}. Since $H_\lambda = G_\lambda\,\Gopt$ and $G_\lambda$ is stable, $\Goptz(z)$ has a single pole at $z=1$ and there are no zeros outside the unit circle, the root locus will be stable for small gains, so (i) holds. When $\lambda=0$, $\widehat H_0(z) = \widehat G_0(z)\,\Goptz(z)$. Since $\widehat G_0(1)=1$ and $\Goptz(z)$ has a pole at $z=1$ and $(z-1)\Goptz(z)$ is stable, we have (ii). When $\lambda \neq 0$, $\widehat G_\lambda(z)$ has two zeros at $z=1$ and $\widehat G_0(z)$ has a single pole at $z=1$, therefore $\widehat H_\lambda(z)$ has a zero at $z=1$ and (iii) holds. Now we examine properness. Note that $\widehat H^{21}(z)\,\widehat H^{12}(z) = \Gconz^{21}(z)\,\Gconz^{12}(z)\,\Goptz(z)$, so strict properness of either term on the right-hand side implies strict properness of the left-hand side. Finally, properness of $\Gconz(z)$ and $\Goptz(z)$ imply properness of $\widehat H(z)$, and strict properness of $\Gconz^{22}(z)$ implies strict properness of $\widehat H^{22}(z)$, so (iv) holds.
\end{proof}

\begin{figure}[b]
  \centering\includegraphics{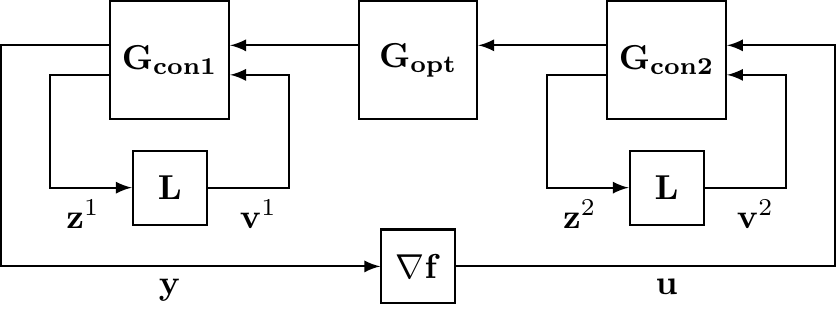}
  \caption{Factored form of an algorithm, where the second-order consensus estimator factors into two first-order SISO estimators; this is the form proposed in \cite{shuo}.}
  \label{fig:factored}
\end{figure}

\begin{remark}
  The continuous-time analog of gradient-based optimization methods are called \emph{gradient flows}, and there has been recent interest in studying iterative algorithms in the continuous limit \cite{su2014differential}. Likewise, consensus methods are often analyzed in continuous time \cite{consensus}. The decomposition described in Theorems~\ref{thm:1}--\ref{thm:2} was developed for discrete-time distributed optimization algorithms, but an analogous decomposition exists for continuous-time systems. In this case, a distributed optimization algorithm would separate into a gradient flow and a continuous-time consensus estimator.
\end{remark}

\subsection{Factoring the consensus estimator}\label{sec:factoring}

The decomposition in Figure \ref{fig:decomp} is not internally stable. While the average gradient is zero at the optimizer, the gradient of each agent is not necessarily zero. This nonzero constant is integrated by the optimization method to produce an unbounded output. This can be fixed, however, if the consensus estimator \textit{factors} into two first-order estimators.

Suppose $\Gcon$ factors as $\Gconleft\Gconright$, where $\Gconleft$ and $\Gconright$ are both first-order estimators. The optimization method and both consensus estimators are SISO LTI systems and therefore commute, so we can swap the order of $\Gopt$ and $\Gconright$ to obtain the diagram in Figure~\ref{fig:factored}. While this does not change the map from $\u$ to $\y$, it does change the realization; the steady-state input to the optimization method is now the average gradient, which is zero at optimality.

To check whether or not a consensus estimator factors, we equate a second-order estimator $\Gcon$ with its factorization $\Gconleft\Gconright$ to find that
\[
  \Gcon = \left[\begin{array}{c:cc}
    \Gconleft^{11}\,\Gconright^{11} & \Gconleft^{12} & \Gconleft^{11}\,\Gconright^{12} \Bstrut\\ \hdashline
    \Gconleft^{21}\,\Gconright^{11} & \Gconleft^{22} & \Gconleft^{21}\,\Gconright^{12} \Tstrut\\
    \Gconright^{21} & 0 & \Gconright^{22} \Tstrut
  \end{array}\right],
\]
where $(z^1,v^1)$ are the transmitted and received variables for $\Gconleft$, and similarly for $\Gconright$. The inputs to the combined system are then $(u,v^1,v^2)$, and the outputs are $(y,z^1,z^2)$. Note that the transmitted variables $v^1$ and $v^2$ need not have the same dimension. The consensus estimator has this form if and only if $\Gcon^{32}$ is zero and its components factor as
\[
  \bmat{\Gcon^{11} & \Gcon^{13} \\ \Gcon^{21} & \Gcon^{23} \Tstrut} = \bmat{\Gconleft^{11} \\ \Gconleft^{21} \Tstrut} \bmat{\Gconright^{11} & \Gconright^{12}},
\]
which is the case if and only if ${\Gcon^{11}\,\Gcon^{23} - \Gcon^{13}\,\Gcon^{21} = 0}$. Whether an estimator factors or not depends on the transfer function $\Gcon$ which is not unique, so we may need to first apply the transformation in Remark~\ref{rem:Ftransform} with a suitable transfer function $\widehat F$ for an estimator to factor.

\subsection{Decomposition of known algorithms}\label{sec:examples}

To illustrate our results, we first describe our decomposition technique on a well-known distributed optimization algorithm. We then state the decomposition for many other algorithms from the literature.

\subsubsection{DIGing}

We first illustrate our results on the DIGing algorithm~\cite{DIGing,LiNa}, which is described by the iterations
\begin{align*}
  \x^{t+1} &= W\x^t - \alpha\,\y^t, \\
  \y^{t+1} &= W\y^t + \df(\x^{t+1}) - \df(\x^t),
\end{align*}
where $\alpha>0$ is the stepsize and the gossip matrix $W$ is related to the graph Laplacian as $W = I-L$. This algorithm requires each agent to communicate $m=2$ variables at each iteration, and the associated transfer function is
\[
  \widehat H(z) = \bmat{\frac{-\alpha}{z-1} & \frac{-z}{z-1} & \frac{-\alpha z}{(z-1)^2} \\
  \frac{-\alpha}{z\,(z-1)} & \frac{-1}{z-1} & \frac{\alpha}{(z-1)^2} \Tstrut\\
  \frac{1}{z} & 0 & \frac{-1}{z-1} \Tstrut}.
\]
Choose the optimization method as ${\Goptz(z) = \widehat H^{11}(z) = \frac{-\alpha}{z-1}}$. Then applying the transformation in Remark~\ref{rem:Ftransform} with the transfer matrix $\widehat F(z) = \text{diag}\bigl(z,\frac{-\alpha z}{z-1}\bigr)$, the consensus estimator transforms as
\[
  \Gconz(z) = \bmat{1 & \frac{-z}{z-1} & \frac{\alpha z}{(z-1)^2} \\
    \frac{1}{z} & \frac{-1}{z-1} & \frac{\alpha}{(z-1)^2} \Tstrut\\
    \frac{z-1}{-\alpha z} & 0 & \frac{-1}{z-1} \Tstrut}
    \mapsto
  \bmat{1 & \frac{-1}{z-1} & \frac{-1}{z-1} \\
    1 & \frac{-1}{z-1} & \frac{-1}{z-1} \Tstrut\\
    1 & 0 & \frac{-1}{z-1} \Tstrut}.
\]
The estimator on the right satisfies the conditions to factor in Sec.~\ref{sec:factoring}; we chose the transformation matrix such that this is the case. Since ${\Gcon^{11} = 1}$, we can choose $\Gconleft^{11} = 1 = \Gconright^{11}$, which results in the factorization $\Gcon = \Gconleft\,\Gconright$, where both factors are the first-order estimator in \eqref{eq:estimator}.

The analysis for all other algorithms in this section is similar. In each case, we choose the optimization algorithm as $\Gopt = H^{11}$ so that $\Gcon^{11} = 1$. In addition, we apply the transformation in Remark~\ref{rem:Ftransform} to put the estimators in a similar form with $\Gcon^{21}=1$ for comparison.

\subsubsection{Non-accelerated algorithms}\label{sec:non-accelerated}

We first consider algorithms that use standard gradient descent for the optimization method: $\Goptz(z) = \frac{-\alpha}{z-1}$ where $\alpha>0$ is the stepsize. Several such algorithms have been proposed whose consensus estimator factors (see Section~\ref{sec:factoring}). In particular, each factor is typically one of the following first-order estimators:
\[
  \Gconleftz(z),\,\Gconrightz(z) = \bmat{1 & \frac{-1}{z-1} \\ 1 & \frac{-1}{z-1} \Tstrut} \quad\text{or}\quad
  \bmat{1 & \frac{-z}{z-1} \\ 1 & \frac{-1}{z-1} \Tstrut}.
\]
Every combination of these factors has been proposed in the literature: DIGing \cite{DIGing,LiNa} uses the estimator on the left for both factors, $\mathcal{AB}$ \cite{AB} uses one of each\footnote{The $\mathcal{AB}$ method is described in terms of two gossip matrices $\mathcal{A}$ and $\mathcal{B}$, where the Laplacian is $L = I-\mathcal{A} = I-\mathcal{B}$.}, and AugDGM \cite{AugDGM} uses the one on the right for both factors.

Not every algorithm uses a consensus estimator that factors into two first-order estimators. To check whether or not an algorithm factors, we search for a transfer matrix $\widehat F$ with full normal rank such that the transformed consensus estimator in Remark~\ref{rem:Ftransform} satisfies the necessary conditions for factorization in Section~\ref{sec:factoring}. Here are the second-order consensus estimators for some algorithms that do not factor:
\begin{align*}
  \Gconz(z) &= \bmat{1 & \frac{-\frac{1}{2} z^2}{(z-1)^2} \\ 1 & \frac{\frac{1}{2}-z}{(z-1)^2} \Tstrut} & \text{Exact Diffusion \cite{ExactDiffusion1}} \\[2mm]
  \Gconz(z) &= \bmat{1 & \frac{-\frac{1}{2} z^2}{(z-1)^2} \\ 1 & \frac{-\frac{1}{2}+z-z^2}{(z-1)^2} \Tstrut} & \text{NIDS \cite{NIDS}} \\[2mm]
  \Gconz(z) &= \bmat{1 & \frac{\frac{1}{2}-z}{(z-1)^2} \\ 1 & \frac{\frac{1}{2}-z}{(z-1)^2} \Tstrut} & \text{EXTRA \cite{EXTRA}} \\[2mm]
  \Gconz(z) &= \bmat{1 & \frac{-z(z+\beta-1)}{(z-1)^2} \\ 1 & \frac{1-(1+\beta)z}{(z-1)^2} \Tstrut} & \text{SVL \cite{SVL}}
\end{align*}

\subsubsection{Accelerated algorithms}

Our decomposition also applies to accelerated algorithms. The optimization method then has the form \cite{tmm,lessard}
\[
  \Goptz(z) = -\alpha\,\frac{(1+\gamma)\,z - \gamma}{(z-1)(z-\beta)},
\]
where $\beta$ and $\gamma$ are additional parameters. Examples include $\mathcal{AB}m$ \cite{ABm} based on the heavy-ball optimization method \cite{polyak1964} with ${\gamma=0}$, and $\mathcal{ABN}$ \cite{ABN} based on Nesterov's accelerated method \cite{nesterov-book} with $\gamma=\beta$. For each of these algorithms, the consensus estimator factors into the two first-order estimators
\[
  \Gconleftz(z) = \bmat{1 & \frac{1}{\alpha} \Goptz(z) \\ 1 & \frac{1}{\alpha} \Goptz(z) \Tstrut}
  \quad\text{and}\quad
  \Gconrightz(z) = \bmat{1 & \frac{-1}{z-1} \\ 1 & \frac{-1}{z-1} \Tstrut}.
\]

\section{Perspectives}

Our decomposition of an algorithm into its optimization and consensus components leads to some perspectives that may prove useful for algorithm design.

\paragraph{Robust optimization}

Using our decomposition, we can interpret an algorithm for distributed optimization as an optimization method that, along with the gradient, includes an additional consensus estimator in the loop. If this consensus estimator were to converge arbitrarily fast, then the iterates would never be in disagreement and the system would reduce to that of the centralized optimization method. Because the consensus estimator is not ideal, however, the optimization method must be \textit{robust} to the dynamics of the estimator; see \cite{rmm,mert,scherer,mihailo} for robust optimization methods.

\paragraph{Consensus with feedback}

Alternatively, we can view an algorithm as a second-order consensus estimator whose input is obtained by feeding back the output through the gradient and the optimization method. In this interpretation, the consensus estimator must be stable when connected in feedback. This feedback loop is linear when the local objective functions are quadratic (gradients are linear), but is otherwise \textit{nonlinear}.


Each of these interpretations provides a certain perspective on the combined algorithm. Ideally, the design of the optimization and consensus components would decouple, enabling researchers to make use of the abundant literature on optimization and consensus. Our decomposition provides a first step towards this decoupling, with these perspectives indicating that proper measures of robustness must be taken into account in the algorithm design.

\bibliographystyle{IEEEtran}
{\small\bibliography{references}}

\begin{thebibliography}{10}
\providecommand{\url}[1]{#1}
\csname url@rmstyle\endcsname
\providecommand{\newblock}{\relax}
\providecommand{\bibinfo}[2]{#2}
\providecommand\BIBentrySTDinterwordspacing{\spaceskip=0pt\relax}
\providecommand\BIBentryALTinterwordstretchfactor{4}
\providecommand\BIBentryALTinterwordspacing{\spaceskip=\fontdimen2\font plus
\BIBentryALTinterwordstretchfactor\fontdimen3\font minus
  \fontdimen4\font\relax}
\providecommand\BIBforeignlanguage[2]{{%
\expandafter\ifx\csname l@#1\endcsname\relax
\typeout{** WARNING: IEEEtran.bst: No hyphenation pattern has been}%
\typeout{** loaded for the language `#1'. Using the pattern for}%
\typeout{** the default language instead.}%
\else
\language=\csname l@#1\endcsname
\fi
#2}}

\bibitem{Unification}
D.~Jakoveti{\'c}, ``A unification and generalization of exact distributed
  first-order methods,'' \emph{IEEE Trans. Sig. Inf. Process. Netw.}, vol.~5,
  no.~1, pp. 31--46, 2018.

\bibitem{EXTRA}
W.~Shi, Q.~Ling, G.~Wu, and W.~Yin, ``{EXTRA}: An exact first-order algorithm
  for decentralized consensus optimization,'' \emph{SIAM J. Optim.}, vol.~25,
  no.~2, pp. 944--966, 2015.

\bibitem{DIGing}
A.~Nedi\'{c}, A.~Olshevsky, and W.~Shi, ``Achieving geometric convergence for
  distributed optimization over time-varying graphs,'' \emph{SIAM J. Optim.},
  vol.~27, no.~4, pp. 2597--2633, 2017.

\bibitem{canform}
A.~Sundararajan, B.~Van~Scoy, and L.~Lessard, ``A canonical form for
  first-order distributed optimization algorithms,'' in \emph{Amer. Contr.
  Conf.}, 2019, pp. 4075--4080.

\bibitem{shuo}
S.~Han, ``Systematic design of decentralized algorithms for consensus
  optimization,'' \emph{IEEE Contr. Syst. Lett.}, vol.~3, no.~4, pp. 966--971,
  2019.

\bibitem{SVL}
A.~Sundararajan, B.~Van~Scoy, and L.~Lessard, ``Analysis and design of
  first-order distributed optimization algorithms over time-varying graphs,''
  \emph{IEEE Trans. Contr. Netw. Syst.}, vol.~7, no.~4, pp. 1597--1608, 2020.

\bibitem{lessard}
L.~Lessard, B.~Recht, and A.~Packard, ``Analysis and design of optimization
  algorithms via integral quadratic constraints,'' \emph{SIAM J. Optim.},
  vol.~26, no.~1, pp. 57--95, 2016.

\bibitem{sundararajan_allerton}
A.~Sundararajan, B.~Hu, and L.~Lessard, ``Robust convergence analysis of
  distributed optimization algorithms,'' in \emph{Allerton Conf. Commun. Contr.
  Comput.}, 2017, pp. 1206--1212.

\bibitem{tmm}
B.~Van~Scoy, R.~A. Freeman, and K.~M. Lynch, ``The fastest known globally
  convergent first-order method for minimizing strongly convex functions,''
  \emph{IEEE Contr. Syst. Lett.}, vol.~2, no.~1, pp. 49--54, 2018.

\bibitem{nesterov-book}
Y.~Nesterov, \emph{Lectures on Convex Optimization}.\hskip 1em plus 0.5em minus
  0.4em\relax Springer Optimization and Its Applications, 2018, vol. 137.

\bibitem{polyak1964}
B.~Polyak, ``Some methods of speeding up the convergence of iteration
  methods,'' \emph{USSR Comput. Math. Math. Phys.}, vol.~4, no.~5, pp. 1--17,
  1964.

\bibitem{consensus}
S.~S. Kia, B.~Van~Scoy, J.~Cort\'{e}s, R.~A. Freeman, K.~M. Lynch, and
  S.~Mart\'{i}nez, ``Tutorial on dynamic average consensus: {T}he problem, its
  applications, and the algorithms,'' \emph{IEEE Contr. Syst. Mag.}, vol.~39,
  no.~3, pp. 40--72, 2019.

\bibitem{dac}
M.~Zhu and S.~Mart\'{i}nez, ``Discrete-time dynamic average consensus,''
  \emph{Automatica}, vol.~46, pp. 322--329, 2010.

\bibitem{fast-consensus}
L.~Xiao and S.~Boyd, ``Fast linear iterations for distributed averaging,''
  \emph{Syst. Control. Lett.}, vol.~53, no.~1, pp. 65--78, 2004.

\bibitem{levine2018control}
W.~S. Levine, \emph{The Control Handbook (three volume set)}.\hskip 1em plus
  0.5em minus 0.4em\relax CRC press, 2018, vol.~1.

\bibitem{plp}
R.~A. Freeman, T.~R. Nelson, and K.~M. Lynch, ``A complete characterization of
  a class of robust linear average consensus protocols,'' in \emph{Amer. Contr.
  Conf.}, 2010, pp. 3198--3203.

\bibitem{su2014differential}
W.~Su, S.~Boyd, and E.~Cand\`es, ``A differential equation for modeling
  {Nesterov's} accelerated gradient method: {T}heory and insights,'' \emph{Adv.
  Neur. Inf. Process. Syst.}, vol.~27, 2014.

\bibitem{LiNa}
G.~Qu and N.~Li, ``Harnessing smoothness to accelerate distributed
  optimization,'' \emph{IEEE Trans. Contr. Netw. Syst.}, vol.~5, no.~3, pp.
  1245--1260, 2017.

\bibitem{AB}
R.~Xin and U.~A. Khan, ``A linear algorithm for optimization over directed
  graphs with geometric convergence,'' \emph{IEEE Contr. Syst. Lett.}, vol.~2,
  no.~3, pp. 315--320, 2018.

\bibitem{AugDGM}
J.~Xu, S.~Zhu, Y.~C. Soh, and L.~Xie, ``Augmented distributed gradient methods
  for multi-agent optimization under uncoordinated constant stepsizes,'' in
  \emph{IEEE Conf. Decis. Contr.}, 2015, pp. 2055--2060.

\bibitem{ExactDiffusion1}
K.~Yuan, B.~Ying, X.~Zhao, and A.~H. Sayed, ``Exact diffusion for distributed
  optimization and learning---{Part I}: Algorithm development,'' \emph{IEEE
  Trans. Sig. Process.}, vol.~67, no.~3, pp. 708--723, 2018.

\bibitem{NIDS}
Z.~Li, W.~Shi, and M.~Yan, ``A decentralized proximal-gradient method with
  network independent step-sizes and separated convergence rates,'' \emph{IEEE
  Trans. Sig. Process.}, vol.~67, no.~17, pp. 4494--4506, 2019.

\bibitem{ABm}
R.~Xin and U.~A. Khan, ``Distributed heavy-ball: {A} generalization and
  acceleration of first-order methods with gradient tracking,'' \emph{IEEE
  Trans. Automat. Contr.}, vol.~65, no.~6, pp. 2627--2633, 2020.

\bibitem{ABN}
R.~Xin, D.~Jakoveti\'{c}, and U.~A. Khan, ``Distributed {N}esterov gradient
  methods over arbitrary graphs,'' \emph{IEEE Sig. Process. Lett.}, vol.~26,
  no.~8, pp. 1247--1251, 2019.

\bibitem{rmm}
S.~Cyrus, B.~Hu, B.~Van~Scoy, and L.~Lessard, ``{A robust accelerated
  optimization algorithm for strongly convex functions},'' in \emph{Amer.
  Contr. Conf.}, June 2018, pp. 1376--1381.

\bibitem{mert}
N.~S. Aybat, A.~Fallah, M.~G\"{u}rb\"{u}zbalaban, and A.~Ozdaglar, ``Robust
  accelerated gradient methods for smooth strongly convex functions,''
  \emph{SIAM J. Optim.}, vol.~30, no.~1, pp. 717--751, 2020.

\bibitem{scherer}
S.~Michalowsky, C.~Scherer, and C.~Ebenbauer, ``Robust and structure exploiting
  optimisation algorithms: an integral quadratic constraint approach,''
  \emph{Int. J. Contr.}, vol.~94, no.~11, pp. 2956--2979, 2020.

\bibitem{mihailo}
H.~Mohammadi, M.~Razaviyayn, and M.~R. Jovanovi{\'c}, ``Robustness of
  accelerated first-order algorithms for strongly convex optimization
  problems,'' \emph{IEEE Trans. Automat. Contr.}, vol.~66, no.~6, pp.
  2480--2495, 2021.

\end{thebibliography}

\end{document}